\documentclass[11pt,reqno, letterpaper,english]{amsart}
\usepackage[left=1in,right=1in,bottom=1in,top=1in]{geometry}
\usepackage{amsmath,amsthm,amssymb}
\usepackage{mathrsfs}
\usepackage[all,cmtip]{xy}
\usepackage{tikz-cd}
\usepackage{cite}
\usepackage{hyperref,cleveref}
\usepackage{enumerate}
\usepackage{graphicx}
\usepackage{color}
\usepackage{colonequals} 
\usepackage{svg} %

\makeatletter
\renewcommand\subsection{\@startsection{subsection}{2}%
  \z@{.5\linespacing\@plus.7\linespacing}{-.5em}%
  {\normalfont\scshape}}
\makeatother

\theoremstyle{plain}
\newtheorem{lem}{Lemma}
\newtheorem{lemma}[lem]{Lemma}

\newtheorem{theorem}[lem]{Theorem}

\newtheorem{proposition}[lem]{Proposition}

\newtheorem{corollary}[lem]{Corollary}
\newtheorem{conj}[lem]{Conjecture}

\theoremstyle{definition}

\newtheorem{definition}[lem]{Definition}
\newtheorem{example}[lem]{Example}
\newtheorem{remark}[lem]{Remark}

\newtheorem{notation}[lem]{Notation}

\numberwithin{equation}{section}
\numberwithin{lem}{section}

\newcommand{\mathfont}{\mathbf}
\newcommand{\R}{\mathfont R}

\newcommand{\Z}{\mathfont Z}

\newcommand{\Q}{\mathfont Q}

\newcommand{\N} {\mathfont N}
\newcommand{\F}{\mathfont F}

\newcommand{\cP}{\mathcal{P}}

\DeclareFontFamily{OT1}{rsfs}{}
\DeclareFontShape{OT1}{rsfs}{n}{it}{<-> rsfs10}{}
\DeclareMathAlphabet{\mathscr}{OT1}{rsfs}{n}{it}

\DeclareMathOperator{\Hom}{Hom}

\DeclareMathOperator{\lcm}{lcm}

\newcommand{\Zp}{\mathfont{Z}_p}

\newcommand{\PP}{\mathfont{P}}

\DeclareMathOperator{\Jac}{Jac}

\DeclareMathOperator{\HH}{H}

\allowdisplaybreaks

\usepackage{float}

\newcommand{\Qpositive}{\Q_{>0}}

\newcommand{\floor}[1]{\left\lfloor #1 \right\rfloor}
\newcommand{\ceil}[1]{\left\lceil #1 \right\rceil}
\newcommand{\fracpart}[1]{\left\{ #1 \right\}}
\newcommand{\paren}[1]{\left( #1 \right)}

\newcommand{\brackets}[1]{\left[ #1 \right]}

\newcommand{\getdigit}[2][k]{\brackets{p^{#1}}\paren{#2}}
\newcommand{\tDelta}{\widetilde{\Delta}}
\newcommand{\tdelta}{\widetilde{\delta}}
\newcommand{\avg}[1]{\overline{#1}}

\title{Higher a-numbers in \texorpdfstring{$\Z_p$}{Zp}-towers via Counting Lattice Points}
\author{Jeremy Booher, Jack Hsieh, Rakesh Rivera, Vincent Tran, James Upton, Carol Wu}

\begin{document} 

\begin{abstract}
    Booher, Cais, Kramer-Miller and Upton study a class of $\Z_p$-towers of curves in characteristic $p$ with ramification controlled by an integer $d$.  In the special case that $d$ divides $p-1$, they prove a formula for the higher $a$-numbers of these curves involving the number of lattice points in a complicated region of the plane. Booher and Cais had previously conjectured that for $n$ sufficiently large the higher $a$-numbers of the $n$th curve are given by formulae of the form $\alpha(n) p^{2n} + \beta(n) p^n + \lambda_r(n) n + \nu(n) $, where $\alpha,\beta,\nu,\lambda_r$ are periodic functions of $n$.  This is an example of a new kind of Iwasawa theory.  We establish this conjecture by carefully studying these lattice points. 
\end{abstract}

\maketitle
\tableofcontents

\section{Introduction}
    Let $k$ be a perfect field of characteristic $p$.  Given a smooth projective curve $X$ over $k$, we are interested in studying the $p$-torsion in the Jacobian of $X$ as a generalization of the $p$-torsion of the class group of the function field $k(X)$.  Note that in characteristic $p$ the group scheme $\Jac(X)[p]$ is not reduced, so is not determined by its $k$-points, while the $k$-points of $\Jac(X)[p]$ are isomorphic to the $p$-torsion in the class group of $k(X)$.  It is natural to investigate whether some version of Iwasawa theory holds for $\Jac(X)[p]$ in $\Z_p$-towers of curves, in particular for the additional geometric  structure of $\Jac(X)[p]$ which cannot be seen using the class group.
    
    A concrete way to do so is to study \emph{higher $a$-numbers}. Consider the vector space $\HH^0(X,\Omega^1_X)$ of regular differentials on \(X\), with dimension \(g_X\) equal to the genus of \(X\). We denote by \(V_X\) the \emph{Cartier operator}, which is a semi-linear endomorphism of \(\HH^0(X,\Omega^1_X)\) (see \cite{Cartier} for its definition and basic properties). Following \cite{bckmu}, we define the higher $a$-numbers of $X$ as follows.

    \begin{definition}
        For a positive integer $r$, we define the $r$th higher $a$-number of $X$ to be
        \[
        a^{r}(X) \colonequals \dim_k \ker \left( V_X^r : \HH^0(X,\Omega^1_X) \to \HH^0(X,\Omega^1_X) \right).
        \]
    \end{definition}

    This reflects (part of) the structure of $\Jac(X)[p]$.  The Dieudonn\'{e} module $\mathbf{D}(\Jac(X)[p])$ is naturally identified with the de Rham cohomology of $X$ \cite{oda}.  The submodule  $\HH^0(X,\Omega^1_X) \subset \HH^1_{\textrm{dR}}(X)$ corresponds to the kernel of the Frobenius $F$ on the Dieudonn\'{e} module  and the Cartier operator corresponds to the Verschiebung $V$.  This is most familiar for the first $a$-number of $X$, where
    \[
    a^1(X)=\dim_{\F_p} \Hom_{\overline{k}}(\alpha_p, \Jac(X)[p]) = \dim_k \left(\ker F \cap \ker V \right) = \dim_k \ker (V_X | \HH^0(X,\Omega^1_X))).
    \]
    
    Inspired by Iwasawa theory and guided by computation, Booher and Cais formulated conjectures that the higher $a$-numbers of ``reasonable'' $\Z_p$-towers of curves over $k$ ``behave regularly'' in the tower \cite{boohercais}.  Booher, Cais, Kramer-Miller, and Upton proved a formula for the higher $a$-numbers in a specific class of reasonable $\Z_p$-towers \cite{bckmu}.  The goal of this paper is to connect the two and establish the conjecture for this class of towers.
    
    We first state the conjecture in this special case.  
    Let $\{X_n\}_{n \geq 0}$ be a $\Z_p$-tower of curves over the projective line, i.e. a sequence of smooth projective curves over $k$ with morphisms
    \[
    \ldots \to X_n \to X_{n-1} \to \ldots \to X_1 \to X_0
    \]
    such that $X_n$ is a branched $\Z/p^n\Z$-cover of $X_0 \simeq \PP^1_k$.  We assume that each $X_n$ is totally ramified over $\infty \in \PP^1_k \simeq X_0$ and unramified elsewhere.  Finally, we assume that the tower has minimal break ratios \cite[Definition 3.2]{bckmu}: if $s_n$ are the breaks in the upper-numbering ramification filtration for the tower above infinity then there is an integer $d$ such that $s_n = d p^{n-1}$ for $n \geq 1$.
    The integer $d$ is the \emph{ramification invariant} of the tower.
    
    \begin{conj}[{Booher-Cais~\cite[Conjecture 3.8]{boohercais}}] \label{conj: bc}
        Let $\{X_n\}$ be a $\Z_p$-tower totally ramified over one point of $X_0 \simeq \PP^1$ with minimal break ratios and ramification invariant $d$.  For each $r \geq 1$, there exists an integer $m_r$ and functions $b_r,\nu_r,\lambda_r: \Z/m_r\Z \to \Q$ such that
        \[
        a^r(X_n) = d \cdot \frac{r(p-1)}{2 (p+1)((p-1)r  + (p+1))} \cdot p^{2n} + b_r(n) \cdot p^n  + \lambda_r(n) \cdot n + \nu_r(n) \quad \text{for } n \gg 0.
        \]
    \end{conj}

    \begin{remark}
        For any \(d\) prime to \(p\), there are many \(\Z_p\)-towers over \(\PP_k^1\) which are totally ramified at \(\infty\) with ramification invariant \(d\). A simple class of examples is provided by the \emph{basic \(\Z_p\)-towers} of \cite[\S 2.3]{boohercais}. \Cref{conj: bc} was inspired by explicit calculations for these basic towers. In these calculations it appeared that $b_r(n)=0$ (cf. Conjecture 1.2 of \cite{boohercais}).  It was not clear to what extent this generalized to other $\Z_p$-towers; see \Cref{thm: main theorem} below.
    \end{remark}

    In this paper, we focus on the case when $d \mid p-1$.  This restriction on the ramification is natural as it guarantees that the $a$-number and Newton polygon of $X_1$ are determined by $d$ \cite{farnellpries,zhu}.  Furthermore, the main theorem in \cite{bckmu} is most precise when $d \mid p-1$.  In particular, we will be able to prove an exact formula establishing Conjecture~\ref{conj: bc} in this case, as opposed to a weaker asymptotic statement (see \cite[Conjecture 3.4]{boohercais} \cite[Corollary 1.2]{bckmu}).
    
More precisely, when $d \mid p-1$, \cite[Theorem 1.1]{bckmu} gives a formula for $a^r(X_n)$ in terms of the number of lattice points in a complicated region in the plane.  
    To state it, we need to set up a bit of notation.
    It relies on ordering integer sequences  $(a_n)_{n\ge 0}$ and $(b_n)_{n\ge 0}$ using the lexicographic ordering:
    \begin{equation} \label{eq:lexicographic}
    (a_n)_{n\ge 0} > (b_n)_{n\ge 0} \text{ if there exists }  m\geq 0 \textrm{ such that }   a_i = b_i \textrm{ for } 0\leq i < m \text{ and } a_m > b_m.
    \end{equation}

    \begin{definition} \label{defn: mu}
        Fix a prime $p$, and $d \mid p-1$.
            For $i \in \Z_{\ge 0}$, let $\displaystyle \frac{p+1}{d} i = \sum_{n\in \Z} i_n p^{n}$ be the (archimedean) base-$p$ expansion of  $\frac{p+1}{d} i$. Then define
            \begin{equation*}
                    \mu(i) 
                    \colonequals 
                    \begin{cases}
                        \floor{ \frac{p+1}{d}i } & \text{if}\ (i_n)_{n\geq 0} > (i_{-1-n})_{n \geq 0}  \\ \\
                        \ceil{ \frac{p+1}{d}i } & \text{otherwise}
                    \end{cases}
            \end{equation*}

    \end{definition}

    \begin{definition} \label{defn: tn}
        Fix $p$, $r$, and $d \mid p-1$.  Let $\gamma \colonequals  \frac{(p-1) r + (p+1)}{d}$.
        For a positive integer $n$, let $t_n \colonequals \displaystyle\floor{\gamma^{-1} p^n} =   \left \lfloor \frac{d p^n}{(r+1)p - (r-1)} \right \rfloor$ and define
        \[
        \Delta_n \colonequals \left \{ (i,j) \in \Z_{>0}^2 : i > t_n \text{ and } \mu(i) \leq j \leq p^n-1 \right \}.
        \]
    \end{definition}

    \begin{theorem}[Booher, Cais, Kramer--Miller, and Upton]    \label{thm: bckmu}
        Let $\{X_n\}$ be a $\Z_p$-tower totally ramified over one point of $X_0 \simeq \PP^1$ with minimal break ratios and ramification invariant $d$.  When $d \mid p-1$, the $r$th higher a-number of $X_n$ is given by
        \[
        a^r(X_n) =  \frac{r (p-1) t_n (t_n+1)}{2d}  + \# \Delta_n.
        \]
    \end{theorem}

    \begin{proof}
        This is \cite[Theorem 1.1]{bckmu}, combined with \cite[Proposition 4.13]{bckmu} to relate the definition of $\mu$ in that work to the one in \Cref{defn: mu}.
    \end{proof}

    In this paper, we connect \Cref{conj: bc} with \Cref{thm: bckmu} by carefully analyzing $\# \Delta_n$. In particular, we prove:

    \begin{theorem} \label{thm: main theorem}
        Let $\{X_n\}$ be a $\Z_p$-tower totally ramified over one point of $X_0 \simeq \PP^1$ with minimal break ratios and ramification invariant $d$.  If $d \mid p-1$ then there exist positive $N_r \in \Z$, $\lambda_r \in \Q$, and a periodic function $\nu_r(n) : \Z_{>0} \to \Q$ such that for $n \geq N_r$
        \begin{equation} \label{eq:maintheorem}
         a^r(X_n) = \frac{r}{r + \frac{p+1}{p-1}} \frac{d}{2(p+1)} p^{2n}+ \lambda_r n + \nu_r(n).
        \end{equation}
        Furthermore, writing $D_r$ for the prime-to-$p$ part of the denominator of $d/(r(p-1)+(p+1))$ in lowest terms,
         the minimal period of $\nu_r$ divides the least common multiple of $2$ and the order of $p$ modulo $D_r$.
    \end{theorem}

    This is a less precise version of \Cref{thm: main v.1}, which provides explicit expressions for $\lambda_r$, $N_r$, and the function $\nu_r(n)$ (and, to an extent, its period).       
    \Cref{thm: main theorem} confirms \Cref{conj: bc} when $d \mid p-1$.  This provides the first theoretical evidence for \Cref{conj: bc} in odd characteristic, as opposed to the weaker asymptotic conjecture established in \cite[Corollary 1.2]{bckmu}. 
Note that it is not obvious that the right side of Equation~\eqref{eq:maintheorem} is an integer.  The main $p^{2n}$ term, which also appears in the asymptotic conjecture, is manifestly not an integer but has a natural interpretation as the area of a triangle appearing in \cite{bckmu} and Section 1.1.  The higher $a$-numbers are clearly integers.  Therefore, to make the right side of \eqref{eq:maintheorem} an integer the rational number $\nu_r(n)$ must be somewhat complicated.

    \begin{example}
        When $p=3$ we can either have $d=1$ or $d=2$.  When $d=2$, \Cref{cor: d1d2} implies that there is a periodic function $\nu_r : \Z \to \Q$ such that
        \[
        a^r(X_n) = \frac{r}{4 (r+2)} 3^{2n} + \nu_r(n).
        \]
        for $n$ sufficiently large, and predicts the minimal period.   For some values of $r$ (i.e. $r=1,2,7$) the function $\nu_r$ is constant.  For others (i.e. $r=3,6$) it has period exactly two.  Other values of $r$ give larger minimal periods.  
        
        Similarly, when $d=1$ there is a periodic $\nu_r : \Z \to \Q$ such that  for sufficiently large $n$
        \[
        a^r(X_n) =  \frac{r}{8 (r+2)} 3^{2n} + \nu_r(n).
        \]
    \end{example}

    \begin{example} \label{ex: intro p5d4}
        When \(p = 5\), \(d=4\), and \(r=2\), we see that
        \[a^2(X_n) = \frac{4}{21} \cdot 5^{2n} + \frac{1}{3}n + \nu_2(n)\]
        where \(\nu_2(n)\) is given in \Cref{table: total p4d4 quasiconstant}.  Note that the values of $\nu_2(n)$ make $4 \cdot 5^{2n} + 7n + 21 \cdot \nu_2(n)$ a multiple of $21$ (necessary since $a^2(X_n)$ is an integer).
         Note that this is a case where the period given by \Cref{thm: main v.1} (6) is a multiple of the observed minimal period (3). 
        \begin{table}[H]
            \begin{tabular}{|l|p{1cm}|p{1cm}|p{1cm}|}
                \hline
                \(n \pmod{3}\)     & \(0\) & \(1\) & \(2\) \\ \hline
                \(\nu_2(n)\)
                & \(-4/21\) & \(-2/21\) & \(2/7\)           \\ \hline
            \end{tabular}
            \caption{Value of periodic function \(\nu_2(n)\) for \(n \geq 0.\)}
            \label{table: total p4d4 quasiconstant}
        \end{table}

These parameters will be used as a recurrent
example throughout the paper (see \Cref{ex: sum floors p5d4,ex: deltasum p5d4,ex: final p5d4}).
Explicit instances of towers
of curves realizing these parameters
 can be constructed using Artin-Schreier-Witt theory from the Witt vector equation 
        \[
        (y_0^p,y_1^p,\ldots,y_{n-1}^p) - (y_0,\ldots,y_{n-1}) =  \sum_{i=1}^d [c_i x^i,0,\ldots,0]
        \]
        where $c_i \in k$ and $c_d \neq 0$.  The curve $X_1$ is easy to describe: it is the Artin-Schreier curve determined by $y_0^p - y_0 = \sum c_i x^i$.  The others are more involved to describe explicitly, see \cite[Remark 3.4 and \S3.2]{bckmu} for a review.
    \end{example}

    \begin{example}
        In contrast, when $p=5$, $d = 4$, and $r=61$ then 
        \[
        a^r(X_n) = \frac{122}{375} 5^{2n} + \frac{2}{3}
        \]
        for $n \geq 3$, but the values for $n=1$ and $n=2$ are not given by the formula.  The more precise \Cref{thm: main v.1} predicts that $N_{61}=3$ (as $v_5(2/125) = -3$), which matches the actual behavior.  
    \end{example}

    \begin{remark}
        The case when $r=1$ is particularly simple, and \cite[\S6]{bckmu} shows (by directly computing $\# \Delta_n$) that if $d \mid p-1$ 
        \[
        a^1(X_n) = \frac{d (p-1)}{4(p+1)} (p^{2n-1}+1) - \begin{cases}
        0 & d \text{ even}\\
        \frac{p-1}{4d} & d \text{ odd}.
        \end{cases}
        \]
        The difficulty of generalizing that argument to other values of $r$ led to a project for students at the PROMYS program, and then to this paper.
    \end{remark}
    
    \begin{remark}
        The period for $a^r(X_n)$ seen in the precise version of \Cref{thm: main theorem} (i.e. \Cref{thm: main v.1}) is similar in spirit but slightly different than the conjectured period in \cite[Conjecture 1.2]{boohercais}.  In particular, the result we prove involves the order of $p$ modulo $D_r$ rather than the order of $p^2$, and also may differ by a factor of $2$.  In particular, the prediction for the period given in \cite[Conjecture 1.2]{boohercais} is not correct.
    \end{remark}

       \begin{remark}
        For any $\Z_p$-tower that is totally ramified over one point of $\PP^1$ and unramified elsewhere, the Deuring-Shafarevich formula shows that the $p$-rank of each curve in the tower is $0$.  Thus the $p$-part of the class group of the function field is trivial.  In contrast, \Cref{thm: main theorem} shows the higher $a$-numbers can grow regularly with $n$ in a non-trivial fashion.
    \end{remark}

    \begin{remark}
    For a curve $X$ of genus $g_X$, the higher $a$-numbers eventually stabilize: $a_X^r = g_X - f_X$ for large enough $r$, where $f_X$ is the $p$-rank of $X$.  More precisely, the $p$-rank is the stable rank of the Cartier operator and the stabilization must happen at least for $r \geq g_X-f_X$ by a simple (semi-)linear algebra argument.  Thus only finitely many of the higher $a$-numbers provide information.  (The smallest integer $r$ for which $a^r_X = a^{r+1}_X$ would be an interesting quantity to study.)
    In contrast, for a $\Zp$-tower as the genus of $X_n$ goes to infinity the higher $a$-numbers $a^r(X_n)$ are eventually interesting for any $r$.
    \end{remark}
    
    \begin{remark}
        The behavior of \'{e}tale $\Z_p$-towers is different, although $\Jac(X_n)[p^\infty]$ still depends on $n$ in a regular manner \cite{caisetale}. Such towers exist only over $\overline{\F}_p$, not over finite fields.  The (necessarily wild) ramification in our setting causes the genus (and higher $a$-numbers) to grow with $n$ at a much faster rate.
    \end{remark}

    \subsection{Idea of the Proof} \label{sec: idea of proof} 
        An elementary argument reinterprets the formula for the higher $a$-numbers  
        in \Cref{thm: bckmu} as the cardinality of
        \[
            \tDelta_n \colonequals \Delta_n \cup  \left \{ (i,j) \in \Z^2_{>0} : i \leq t_n \, \text{and} \, p^n - i r \frac{p-1}{d} \leq j \leq p^n -1 \right \}.
        \]
        The underlying idea in the proof is that $\# \tDelta_n$ is approximately the number of lattice points in the triangle $\cP_n$ bounded by $y = p^n$, $y = \frac{p+1}{d} x$, and $y = p^n - r \frac{p-1}{d}  x$.  See \Cref{fig: diagram} for an illustration.  If we let $\tau = \frac{p+1}{d}$ and $\gamma = \frac{p-1}{d} r +\tau$, then $p^n \gamma^{-1}$ is the $x$-coordinate of the lower vertex of the triangle (explaining the quantity $t_n$) and $p^n \tau^{-1}$ is the $x$-coordinate of the right vertex.
    
        \begin{figure}[ht]  
            \includegraphics[height=3in]{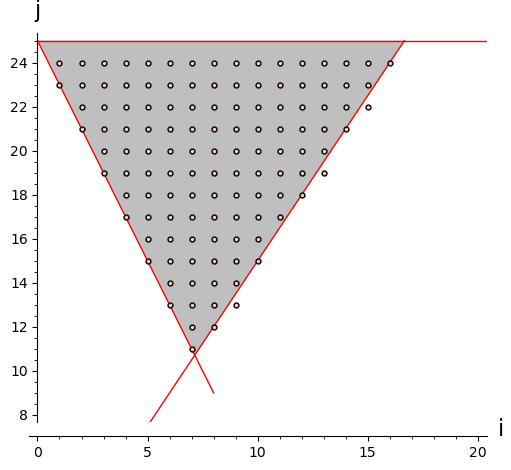}
            \caption{$\tDelta_2$ and $\cP_2$ (shaded) when $p=5$, $r=2$, and $d=4$.}
            \label{fig: diagram}
        \end{figure}

        The key difficulty is that there is irregular behavior along the side of the triangle given by $y = \frac{p+1}{d} x$: this arises from the condition $\mu(i) \leq j$ in the definition of $\Delta_n$.  While $\mu(i)$ is approximately equal to $\frac{p+1}{d} i$, it is subtle to determine whether to round up or down in \Cref{defn: mu}.  This controls whether the first lattice point with $x=i$ below the line $y = \frac{p+1}{d} x$ should be included in $\Delta_n$ or not.  The blue line in Figure~\ref{fig: diagram2} shows the ragged lower right boundary on $\Delta_n$ given by $\mu(i)$.

        \begin{figure}[ht]  
            \includegraphics[height=3in]{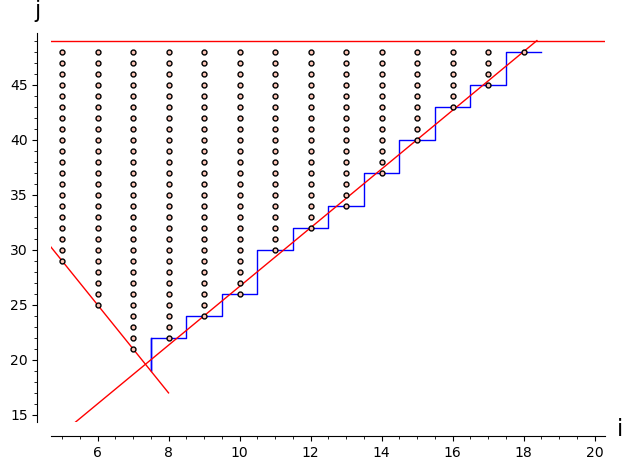}
            \caption{$\tDelta_2$ and $\cP_2$, showing $\mu_i$, when $p=7$, $r=2$, and $d=3$.}
            \label{fig: diagram2}
        \end{figure}       
        
        \Cref{subsec: basep and periodic} establishes properties of base-$p$ rationals and periodic functions ubiquitous throughout the following sections. In \Cref{sec: delta prop}, we define an indicator function $\delta(i)$ (taking on values $0$ and $1$) which reflects whether we round up or down in \Cref{defn: mu}.  The key insight is that $\delta$ is a periodic function and that we can compute its average value over a period.  This gives control over sums of the form $\sum_{i=m}^n \delta(i)$, which reflect the number of lattice points in \Cref{fig: diagram} that lie outside the triangle.
        In \Cref{sec: rewriting sum}, we essentially express higher $a$-numbers in terms of lattice points in $\cP_n$ and sums involving $\delta$.  The following version is most convenient: 
        \[
        a^r(X_n) =   \left(  \sum_{i=1 }^{\floor{\tau^{-1} p^n}} 
            (p^n - \floor{\tau i}) 
            -
            \sum_{i=1}^{\floor{\gamma^{-1} p^n}}  (p^n - \floor{\gamma i}) \right) - \left(\sum_{i=1}^{\floor{\tau^{-1} p^n}} \delta(i) - \sum_{i=1}^{\floor{\gamma^{-1} p^n}} \delta(i) \right).
        \]
        The sums involving floor functions count the number of lattice points in a variant of $\cP_n$: see \Cref{remark:boundary issues}.
        
        As discussed in \Cref{sec:erhart}, the theory of Ehrhart polynomials gives a natural approach to counting the number of lattice points in $\cP_n$, since \(\cP_n\) is a triangle with rational vertices and is the triangle \(\cP_0\) scaled-up by a factor of $p^n$. However, in \Cref{sec: sums floor} we give a direct proof using the same ideas as we use to study sums involving $\delta$ in \Cref{sec: sums delta}.  The basic idea is to exploit the periodicity of the function being summed and compare with the length of the sum.  The latter depends on $p^n$, which explains the appearance of the order of $p$ modulo $D_r$ in \Cref{thm: main theorem}.  We put everything together in \Cref{sec: main thm}, and show that the $p^n$-term in the various pieces cancel out in the resulting formula.
    
        In \Cref{sec: additional}, we prove additional properties about $\lambda_r$ and the minimal possible integer $N_r$ appearing in \Cref{thm: main theorem}. We end with some further properties of $a^r(X_n)$ we have observed from empirical data, but have been unable to prove. 

    \subsection{Acknowledgements}
We thank Bryden Cais and Joe Kramer-Miller for many helpful conversations over the course of this project.
We would like to thank the PROMYS program for running these research lab projects, with support from the PROMYS Foundation and the Clay Mathematics Institute. We also thank Jonah Mendel for mentoring the group at PROMYS in Boston.  Finally, we thank 
students Praveen Balakrishnan and Bertie Parkes and counselor Mario Marcos Losada at PROMYS Europe for helpful conversations about this project.

\section{Preliminaries}

\subsection{Notation}
 Fix an odd prime $p$, a positive integer $d$ dividing $p-1$, and a positive integer $r$.
    
    \begin{notation} \label{defn: tau and gamma}
         We define 
        \begin{enumerate}[(i)]
            \item $\tau = \frac{p+1}{d}$
            \item \(\gamma = \frac{p-1}{d}r + \tau = \frac{(p-1)r + (p+1)}{d}\).
        \end{enumerate}
    \end{notation}
    
    Note that, by hypothesis, \(\gamma - \tau\) is integral and \(\tau, \gamma\) are positive. Furthermore, recalling \Cref{defn: tn}, we have \(t_n =  \floor{\gamma^{-1} p^n}\).
    
        \begin{notation} \label{defn: ND notation}
            Let $x$ be a positive rational number. 
            We let $x_N,x_D$ be the unique coprime positive integers such that 
            \begin{equation*}
                x = p^{v_p(x)} \frac{x_N}{x_D},
            \end{equation*}
            where as usual $v_p(x)$ is the $p$-adic valuation of $x$.
        \end{notation}
    
    Since $d \mid p-1$ and hence \(\gcd(d, p+1) = \gcd(d,2)\), we see that 
     \(\tau_D = \frac{d}{\gcd(d,2)}\) and  \(\tau_N = \frac{p+1}{\gcd(d,2)}\).

We end with a list of references for the other notation defined elsewhere. 

     \begin{itemize}
         \item The set of lattice points $\Delta_n$ is defined in \Cref{defn: tn}, and its variant $\widetilde{\Delta}_n$ in \Cref{defn:  tDelta}.  The functions $\mu(i)$ and $\delta(i)$
    describe the boundary of $\Delta_n$, and are defined in \Cref{defn: mu} and \Cref{defn: deltas}.  The latter also defines the variants $\delta_0(i)$ and $\widetilde{\delta}(i)$.

    \item  For a periodic function $f$, $\overline{f}$ denotes the average value (\Cref{defn: periodic & average value}).

    \item  For a positive rational number $x$, $\overline{x}$ denotes the average value of the repeating part of its base-$p$ expansion (\Cref{defn: Lx and Dx}).  The period of its repeating part is denoted $L_x$, and $\getdigit[k]{x}$ is the $k$th digit in the base-$p$ expansion.  The fractional part of $x$ is $\fracpart{x}$.

    \item  For a positive rational number $x$, the functions $A_{x^{-1}}$, $B_{x^{-1}}$, and $F_{x^{-1}}$ appear over the course of Section~\ref{sec:computinghigheranumbers} in \Cref{prop: single floor sum}, \Cref{prop: delta sum}, and \Cref{lemma: delta sum for exponent}.

    \item  The exact formula for $\lambda_r$ appears in \Cref{thm: main v.1}.
         
     \end{itemize}

    \subsection{Base-\texorpdfstring{$p$}{p} expansions and Periodic Sequences} \label{subsec: basep and periodic}
    
        \begin{definition} \label{defn: periodic & average value}
            Let $(f_i)_{i \geq 1}$ be a sequence of real numbers. We say that $f$ is \emph{periodic} for \(i \geq D\) if, for some $L \in \mathbb{Z}$ and all $i \geq D$, we have
            \begin{equation*}
                f_{i+L} = f_i.
            \end{equation*}
            We will also sometimes less precisely call such functions \emph{eventually periodic}.
            In this case we define the \emph{average value} of $f$ over its period to be
            \begin{equation*}
                \avg{f} \text{ or } \avg{f_i} \colonequals \frac{1}{L} \sum_{i = D + 1}^{D+L} f_i.
            \end{equation*}
        \end{definition}

    Note that the average value of a periodic function is independent of the choice of $L$ and $D$.
        We may equivalently view the sequence as a function $f: \N \to \R$.  It is harmless to start indexing sequences at $0$ instead of $1$ if convenient.
        We record the following well-known result; the proof is omitted.
    
        \begin{lemma}[Sum of a periodic sequence]   \label{lemma: periodic sum}
            Let $f = (f_i)_{i \in \N}$ be a periodic sequence with period $L$. For any $N \geq 0$, we have
            \[
            \sum_{i=1}^N f_i = N \avg{f} + \sum_{k=1}^{N \bmod L} (f_k - \avg{f}).
            \] 
        \end{lemma}

        As many of the sums we will deal with are only eventually periodic, it is useful to also have a formula for a sequence that is periodic after a delay.
        
        \begin{corollary}[Sum of an eventually periodic sequence]   \label{cor: eventually periodic sum}
            Let $(f_i)_{i \in \N}$ be periodic for \(i \geq D\) with period $L$. For \(N \geq D - 1\),
            \[\sum_{i=1}^N f_i = \sum_{i=1}^{D-1} f_i \quad+ \paren{(N - D + 1)\avg{f} \quad+ \sum_{k=1}^{(N - D + 1) \bmod L} (f_{k + D - 1} - \avg{f} )}.\]
        \end{corollary}
        
        \begin{proof}
            Noting that
            \[
                \sum_{i=1}^N f_i = 
                \sum_{i=1}^{D-1} f_i + \sum_{i=D}^N f_i,
            \]
            apply \Cref{lemma: periodic sum} to the sequence $f_{D},f_{D+1},\dots$, which by hypothesis is periodic with period $L$.
        \end{proof}
    
        \begin{definition}
            Let $x$ be a positive rational number. For any integer $k \in \Z$, we define $\getdigit[k]{x}$ to be the unique integer in $\{0,\dots,p-1\}$ such that
        \begin{equation*}
            \getdigit[k]{x} \equiv \left \lfloor \frac{x}{p^k} \right\rfloor \pmod{p}. 
        \end{equation*}
        \end{definition}
        
        In other words, $\getdigit[k]{x}$ is the $k$th digit of the base-$p$ expansion of $x$, so that
        \begin{equation*}
            x = \sum_{k \in \Z} \getdigit[k]{x} p^k.
        \end{equation*}
        Since $x$ is rational, the sequence $(\getdigit[-k]{x})_{k \in \N}$ is eventually periodic \cite[Theorem 136]{hardywright}.
    
        \begin{definition}  \label{defn: Lx and Dx}
            Let $x$ be a positive rational number. We set $L_x$ to be the minimal period of the eventually periodic sequence $\getdigit[-k]{x}$.  We define the \emph{delay} $D_x$ to be the minimal $D$ such that the sequence is periodic for $k \geq D + 1$.  We also set $\avg{x} = \avg{\getdigit[-k]{x}}$, where the bar is used to indicate the average value as in \Cref{defn: periodic & average value}.
        \end{definition}

        \begin{example}
        In base-$5$
        \[
        \frac{2}{3} = .3131\ldots \quad \text{and} \quad \frac{2}{7} = .120324120324\ldots
        \]
        Thus $L_{2/3} = 2$, $D_{2/3} = 0$, and $\avg{2/3} = (3+1)/2 = 2$.  Similarly $L_{2/7} = 6$, $D_{2/7}=0$, and $\avg{2/7} = 2$.  On the other hand, the (base-$10$) rational number $x = 2803/672$ has base-$7$ expansion $4.1124612461\ldots$ with delay $D_{x}=1$ and period $L_x=4$.
        \end{example}

        Recall that $x_D$ is defined in \Cref{defn: ND notation}.

        \begin{lemma}   \label{lemma: Lx}
           Let $x$ be a positive rational number. Then $L_x$ is the order of $p$ in $(\Z/x_D\Z)^\times$ and $D_x =  -v_p(x)$.
        \end{lemma}
    
        \begin{proof}
        This is well-known, see for example \cite[Thm 136]{hardywright}.
        \end{proof}
    
        \begin{corollary}   \label{cor: period1}
            For any positive integer $m$, $\frac{m}{d}$ has a periodic base-$p$ expansion with period $1$.
        \end{corollary}
    
        \begin{proof}
            Let $x = \frac{m}{d}$, so $x_D$ divides $d$ which divides $p-1$.  Thus $p \equiv 1 \pmod{x_D}$ has order one.
        \end{proof}
    
    \subsection{Properties of \texorpdfstring{$\delta$}{\textbackslash delta}}  \label{sec: delta prop}
        Our next goal is to study when the two cases in \Cref{defn: mu} occur.  We make use of the lexicographic ordering on sequences (recall Equation~\eqref{eq:lexicographic}).

        \begin{definition}  \label{defn: deltas}
            We define a function \(\delta: \N \to \{0, 1\}\) by
            \begin{equation*}
                \delta(i) \colonequals 
                \begin{cases}
                    0 & 
                    \text{if}\ (\getdigit[n]{\tau i})_{n\geq 0} > (\getdigit[-1-n]{\tau i})_{n \geq 0} %
                \\
                    1 & \text{otherwise}.\\
                \end{cases}
            \end{equation*}

            We define variants $\delta_0, \tdelta : \N \to \{0,1\}$ by
            \[
            \delta_0(i) \colonequals 
            \begin{cases}
                0 & \text{if } p \mid i    \\
                \delta(i) & \text{otherwise}
            \end{cases} \quad \text{and} \quad
            \tdelta(i) \colonequals 
            \begin{cases}
                1 & \text{if } \tau_D \mid i \\
                \delta(i) & \text{otherwise}.
            \end{cases}     
            \]
        \end{definition}

        Recalling \Cref{defn: mu}, it is immediate that $\mu(i) = \floor{\frac{p+1}{d} i} + \delta(i)$.  Note that if $\frac{p+1}{d}i$ is an integer (i.e. $\tau_D \mid  i$) we have $\delta(i)=0$ while $\tdelta(i) = 1$, so we may similarly write $\mu(i) = \left \lceil \frac{p+1}{d} i \right \rceil  - 1  + \tdelta(i)$.  
    
        \begin{example}
            \Cref{table: values delta} shows $\delta(i)$ and $\delta_0(i)$ for small $i$ when $p=5$ and $d=4$.  In this case $\tau_D= 2$.
        \end{example}
    
        \begin{table}[ht]
            \begin{tabular}{l|lllllllllllllllllll}
                $i$ & $1$ & $2$ & $3$ & $4$ & $5$ & $6$ & $7$ & $8$ & $9$ & $10$ & $11$ & $12$ & $13$ & $14$ & $15$ & $16$ & $17$ & $18$ & $19$ \\ \hline 
                $\delta(i)$ & $1$ & $0$ & $0$ & $0$ & $1$ & $0$ & $1$ & $0$ & $0$ & $0$ & $1$ & $0$ & $0$ & $0$ & $0$ & $0$ & $1$ & $0$ & $0$ \\
                $\delta_0(i)$ &$1$ & $0$ & $0$ & $0$ & $0$ & $0$ & $1$ & $0$ & $0$ & $0$ & $1$ & $0$ & $0$ & $0$ & $0$ & $0$ & $1$ & $0$ & $0$ \\
                $\tdelta(i)$ & $1$ & $1$ & $0$ & $1$ & $1$ & $1$ & $1$ & $1$ & $0$ & $1$ & $1$ & $1$ & $0$ & $1$ & $0$ & $1$ & $1$ & $1$ & $0$ \\
            \end{tabular}
            \caption{Values of $\delta(i)$ with $d=4$, $p=5$.}
            \label{table: values delta}
        \end{table}
    
        We now turn to studying $\delta(i)$.
        
        \begin{lemma}   \label{lemma: digit comparison for multiple of p}
            For a positive integer $i$, we have that  \(\getdigit[-1]{\tau i} = \getdigit[0]{\tau i}\) if and only if  $p \mid  i$ .
        \end{lemma}
    
        \begin{proof}
            Since $d\mid p-1$, we know from \Cref{cor: period1} that the base-$p$ expansion of $\tau i$ consists of a single repeating digit $h$.  Then we see the fractional part of $\tau i$ is \(\fracpart{\tau i} = 0.\bar{h} = \frac{h}{p-1}\) and 
            hence \(\getdigit[0]{\tau i} \equiv \left(\tau i - \frac{h}{p-1} \right) \pmod{p}\).
                We then see
                \begin{alignat*}{2}
                    \getdigit[0]{\tau i} = \getdigit[-1]{\tau i}
                    &\iff
                    \tau i - \frac{h}{p-1} &&\equiv h \pmod p
                \\
                    &\iff
                    \tau i + h &&\equiv h \pmod p
                \\
                    &\iff
                    \tau i &&\equiv 0 \pmod p
                \\
                    &\iff 
                    i &&\equiv 0 \pmod{p}
                \end{alignat*}
            where the last step uses that $p+1$ and $d$ are units modulo $p$.  
        \end{proof}
            
        \begin{lemma}   \label{lemma: digit comparison for non-multiple of p}
            Suppose \(p \: \nmid \: i\). Then \(\getdigit[-1]{\tau i} > \getdigit[0]{\tau i} \iff \delta(i) = 1\).
        \end{lemma}
            
        \begin{proof}
            As \(\getdigit[-1]{\tau i} \ne \getdigit[0]{\tau i}\) by \Cref{lemma: digit comparison for multiple of p}, we see that 
            \begin{align*}
                \getdigit[-1]{\tau i} > \getdigit[0]{\tau i} &\iff
                (i_n)_{n \geq 0} \leq (i_{-1-n})_{n \geq 0}
            \\
                &\iff \delta(i) = 1. \qedhere
            \end{align*}
        \end{proof}

              \begin{proposition} \label{prop: delta negation reduction}
            For an integer $i$ with $0 \leq i < \tau_Dp$ and $p \nmid i$:
            \begin{enumerate}[(i)]
                \item  if $\tau_D \nmid i$ then \(\delta(i) + \delta(\tau_Dp- i) = 1\). 
                \item  if $\tau_D \mid  i$ then $\delta(i) = \delta(\tau_Dp-i) =0$.
            \end{enumerate}
            For an integer $j$ with $0 \leq j < d$ and $p \nmid j$:
            \begin{enumerate}[(i)]      \setcounter{enumi}{2}
                \item if $\tau_D \nmid j$ then \( \delta(j) + \delta(d -j) = 1 \).
                \item  if $\tau_D \mid  j$ then $\delta(j) =0$.
            \end{enumerate}
        \end{proposition}
        
        \begin{proof}
            We begin with the first statement.  Note that $\tau(\tau_Dp-i) = \tau_Np-\tau i$. Because $\tau_Np$ is integral and $\tau i \notin \mathbb{Z}$, when computing $\tau_Np-\tau i$ we must borrow from the ones place and so
            \begin{align*}
                \getdigit[-1]{\tau_Np-\tau i} &= (p-1) - \getdigit[-1]{\tau i}.
            \end{align*}
            Similarly, as $\tau_N p$ is a multiple of $p$ we must borrow when computing $\tau_Np - \tau i$ so
            \[\getdigit[0]{\tau_Np-\tau i} = (p-1) - \getdigit[0]{\tau i}.\]
     
            As \(p \nmid i\) and \(p \nmid \tau_Dp-i\), we may apply \Cref{lemma: digit comparison for non-multiple of p} to see that \(\delta(i)\) and \(\delta(\tau_Dp -i)\) depend only on these two calculated digits.   Hence $\delta(\tau_Dp-i)=1$ if and only if $(p-1) - \getdigit[-1]{\tau i} > (p-1) - \getdigit[-1]{\tau i}$ if and only if $\delta(i)=0$.
            
            The third statement is similar.  Note that $\tau j + \tau (d-j) = (p+1)$.  If $\tau_D \nmid j$, then $\tau j$ and $\tau (d-j)$ are not integers and so consist of a single non-zero repeating digit after the radix point.  Thus adding digits and keeping track of carries we see that
            \[
            \getdigit[0]{\tau j} + \getdigit[0]{\tau(d-j)} = p \quad \text{and} \quad \getdigit[-1]{\tau j} + \getdigit[-1]{\tau(d-j)} = p.
            \]
            We conclude using \Cref{lemma: digit comparison for non-multiple of p}.  
    
            Finally, if $\tau_D \mid  i$ and $i>0$ then $\delta(i) =0$ since the $\tau i = \frac{\tau_N}{\tau_D} i$ is an integer and hence has $\getdigit[-1]{\tau i} = 0$ while $\getdigit[0]{\tau i } >0$.
        \end{proof}

        \begin{theorem} \label{thm: delta}
            Let $i$ and $n$ be positive integers, and $\tau_D$ be the denominator of $\frac{p+1}{d}$ when written in lowest terms.  Then we have:
            \begin{enumerate}[(i)]
                \item \label{thmparti} $\delta(p^n i) = \delta(i)$.
                \item \label{thmpartii} $\delta_0(i) = \delta_0(i + \tau_D p)$.
                \item  \label{thmpartiii} if $\tau_D \nmid i$ and $p \nmid i$ then $\delta_0(i) + \delta_0(\tau_Dp-i) = 1$.
                \item \label{thmpartiv} if $\tau_D \mid  i$ or $p \mid i$, then $\delta_0(i) = \delta_0(\tau_Dp-i) = 0$.
            \end{enumerate}
        \end{theorem}

        \begin{proof}
            (\ref{thmparti}) By \Cref{cor: period1}, the base-$p$ expansions of $\tau i$ and $\tau p^n i$ consist of a single repeating digit after the radix point, and that is the same digit $h$ for both.  Thus $\delta(i)=1$ if and only if $\getdigit[\ell]{\tau i} \leq h$ for all $\ell \geq 0$.   Similarly, $\delta(p^n i)=1$ if and only if $\getdigit[\ell]{\tau i p^n} = \getdigit[\ell-n]{\tau i}\leq h$ for $\ell \geq 0$.  These two conditions are easily seen to be equivalent as $\getdigit[\ell-n]{\tau i} = h$ for $\ell < n$. 

            (\ref{thmpartii})
            As $\tau i$ and $\tau(i+\tau_Dp) = \tau i + \tau_N p$ have the same digits to the right of the radix point, we see $\getdigit[-1]{\tau(i + \tau_D p)} = \getdigit[-1]{\tau i}$.  
            Similarly, we see that $\getdigit[0]{\tau(i + \tau_Dp)} = \getdigit[0]{\tau i}$ as adding $\tau_N p$ (a multiple of $p$) does not change the digit in the ones place in the base-$p$ expansion of $\tau i$.  We may then apply \Cref{lemma: digit comparison for non-multiple of p} to complete the proof, as $p \nmid i$ and $p \nmid \tau (i + \tau_D p) = \tau i + \tau_N p$.

            Recall that $\delta(i) = \delta_0(i)$ unless $i$ is a multiple of $p$, in which case $\delta_0(i)=0$. Then (\ref{thmpartiii}) and (\ref{thmpartiv}) are restatements of \Cref{prop: delta negation reduction}.
        \end{proof}

        In particular, $\delta_0$ is periodic with period $\tau_D p$ so we may apply the ideas of \Cref{subsec: basep and periodic}.

        \begin{proposition} \label{prop: number of subpalindromics in a tau_Dp interval}
            We have
            \(\avg{\delta_0} = \frac{1}{2}\paren{1 - \frac{1}{p}}\paren{1 - \frac{1}{\tau_D}}.\)
        \end{proposition}
        
        \begin{proof}
            The proposition is a consequence of \Cref{thm: delta}.  
            As $\delta_0$ is immediately periodic\textemdash meaning the delay $D_{\delta_0}$ is zero \textemdash with period $\tau_Dp$, it suffices to sum over the interval from $1$ to $\tau_Dp$.  
            If \(\tau_D \mid i\), then  \(\delta_0(i) = 0\).
            Thus it suffices to consider those $i$ in $[1,\tau_Dp]$ such that $\tau_D \nmid i$ and $p \nmid i$.  
            There are \((p-1)(\tau_D - 1)\) such \(i\), and exactly one of $i , \tau_Dp - i$ makes $\delta_0$ equal to $1$.
            Thus we have \(\sum_{i=1}^{\tau_D p} \delta_0(i) = \frac{(p-1)(\tau_D-1)}{2} \). Dividing by the period \(\tau_D p\) gives the proposition.
        \end{proof}
        
        We will return to these properties to compute more involved sums in \Cref{sec: sums delta}.
       
    \subsection{Simplifying Higher \texorpdfstring{$a$}{a}-Numbers} \label{sec: rewriting sum} 

        We now turn our attention to obtaining an alternate version of the formula for higher $a$-numbers given by \Cref{thm: bckmu} and expressing it using $\delta(i)$.

        \begin{definition}  \label{defn:  tDelta}
            For fixed $p, r$, and $d\mid p-1$ as usual, for a positive integer $n$ define 
            \[
            \tDelta_n \colonequals \Delta_n \cup  \left \{ (i,j) \in \Z^2_{>0} : i \leq t_n \, \text{and} \, p^n - i r \frac{p-1}{d} \leq j \leq p^n -1 \right \}.
            \]
        \end{definition}

        \begin{lemma}   \label{lemma: rewrite tDelta}
            Continuing the hypotheses and notation of \Cref{thm: bckmu}, $a^r(X_n) = \# \tDelta_n$.
        \end{lemma}
        
        \begin{proof}
            For $i \leq t_n$, the number of elements of $\tDelta_n$ with first coordinate $i$ is $i r \frac{p-1}{d}$.  The claim is then a restatement of \Cref{thm: bckmu} as
            \[
            \sum_{i=1}^{t_n} i r \frac{p-1}{d} = \frac{(p-1) r t_n (t_n+1)}{2d}. \qedhere
            \]
        \end{proof}

 \begin{proposition} \label{prop: new sum breakdown}
            With notation as before,
            \begin{align*}
                a^r(X_n) 
                &=
                \sum_{i=1}^{\floor{\gamma^{-1} p^n}}  (\gamma - \tau) i  + \sum_{i=\floor{\gamma^{-1} p^n} +1 }^{\floor{\tau^{-1} p^n}} \left( p^n - \floor{\tau i } \right) -   \sum_{i=\floor{\gamma^{-1} p^n} +1 }^{\lfloor  \tau^{-1} p^n \rfloor} \delta(i)\\
                &=  
                \left(  \sum_{i=1 }^{\floor{\tau^{-1} p^n}} 
                    (p^n - \floor{\tau i}) 
                    -
                    \sum_{i=1}^{\floor{\gamma^{-1} p^n}}  (p^n - \floor{\gamma i}) \right) - \left(\sum_{i=1}^{\floor{\tau^{-1} p^n}} \delta(i) - \sum_{i=1}^{\floor{\gamma^{-1} p^n}} \delta(i) \right).
            \end{align*}
        \end{proposition}
        
        \begin{proof}
            Recall that by definition $t_n = \lfloor \gamma^{-1} p^n \rfloor$, and notice that the largest potential first coordinate of a point in $\Delta_n$ is $\floor{\tau^{-1} p^n}$. Recalling \Cref{defn: tDelta,defn: tn}, we see that for $1 \leq i \leq t_n$, the number of elements of $\tDelta_n$ with first coordinate $i$ is
            $i r \frac{p-1}{d} = (\gamma-\tau) i$.  Similarly, for $t_n < i$, the number of elements of $\Delta_n$ with first coordinate $i$ is $p^n - \mu(i)$.  As $\mu(i) = \floor{\tau i} + \delta(i)$, we conclude that
            \[
            \# \tDelta_n =  \sum_{i=1}^{\floor{\gamma^{-1} p^n}}  (\gamma - \tau) i  + \sum_{i=\floor{\gamma^{-1} p^n} +1 }^{\floor{\tau^{-1} p^n}} \left( p^n - \floor{\tau i} \right)  -   \sum_{i=\floor{\gamma^{-1} p^n} +1 }^{\lfloor  \tau^{-1} p^n \rfloor} \delta(i).
            \]
            The first equality follows using \Cref{lemma: rewrite tDelta}.  The second equality is elementary, noticing that 
            $\gamma - \tau \in \Z_{> 0}$ so $(p^n-\floor{\tau i}) - (p^n - \floor{\gamma i}) = (\gamma-\tau)i$.
        \end{proof}

\subsection{Connection with Ehrhart quasi-polynomials} \label{sec:erhart}
We end by explaining a slightly different perspective, using the theory of Ehrhart quasi-polynomials for rational polytopes, which is not needed for our approach.
       The following triangle is closely related to the region $\Delta_n$:
       
        \begin{definition}
            Let $\cP_n \colonequals p^n \cP_0$ be the triangle with vertices
            $(0,p^n)$, $(\tau^{-1} p^n  , p^n)$, and $(\gamma^{-1} p^n ,  \tau \gamma^{-1} p^n)$.  In particular, $\cP_0$ is the triangle with vertices $(0,1)$, $(\tau^{-1},1)$, and $(\gamma^{-1},\tau \gamma^{-1})$. 
        \end{definition}

        It is an easy exercise to check that $\cP_n$ can equivalently be described as the triangle bounded by the lines $y = p^n$, $y = \tau x = \frac{p+1}{d} x$, and $y = p^n - r \frac{p-1}{d}  x = p^n  - (\gamma-\tau)  x
        $.  An example is shown in \Cref{fig: diagram}: the points on or below $y = \tau x$ are explained by those $i$ where $\delta(i) =0$ as shown in \Cref{table: values delta}.  

        \begin{remark} \label{remark:boundary issues}
        Letting $L(\cP_n)$ denote the number of lattice points in $\cP_n$, we will verify that
            \begin{equation} \label{eq:alternateformula}
            a^r(X_n) = \# \tDelta_n = L(\cP_n) - \lfloor \tau^{-1} p^n  \rfloor  -1 + \sum_{i=t_n+1}^{\lfloor  \tau^{-1} p^n \rfloor} (1- \tdelta(i)) .
            \end{equation}
        The set of lattice points in $\cP_n$ (including points on the boundary) is almost $\tDelta_n$, except:
            \begin{itemize}
                \item  the $\lfloor  \tau^{-1} p^n \rfloor +1$ points $(0,p^n), (1,p^n), \ldots, (\lfloor  \tau^{-1} p^n \rfloor ,p^n)$ are lattice points of $\cP_n$ but not included in $\tDelta_n$,
                
                \item  for $i> t_n$, the lower bound on the possible second coordinates are $\frac{p+1}{d} i$ versus $\mu_i$.  The point $(i,\mu(i)) \in \tDelta_n$ is not in $\cP_n$ if and only if $\mu(i) < \frac{p+1}{d}i $.  Note this happens if and only if $\delta(i) = 0$ and $\frac{p+1}{d} i$ is not an integer (i.e. $\tau_D \nmid i$).  We may use $\tdelta$ to encapsulate this as the condition $\tdelta(i)=0$,
            
                \item  for points $(i,j) \in \Delta_n$ with $i> t_n$ and $j > \mu(i)$, it is automatic that $j \geq \frac{p+1}{d} i$.  
            \end{itemize}
            Thus $\cP_n$ contains $\lfloor \tau^{-1}  p^n  \rfloor +1$ lattice points not appearing in $\tDelta_n$ and misses 
            \( \displaystyle \sum_{i=t_n+1}^{\lfloor \tau^{-1} p^n  \rfloor}  (1-\tdelta(i)) \)
            elements of $\Delta_n$. 

      Now the triangle $\cP_n$ is the dilation of $\cP_0$ by $p^n$, and the theory of Ehrhart quasi-polynomials describes the number of lattice points in the dilation. In particular,
the main result of \cite{mcmullen} implies that  the Ehrhart quasi-polynomial of the triangle 
                $\cP_0$ is
                \[
                        \frac{1}{2}(\tau^{-1}-\gamma^{-1}) t^2
                        + \frac{1}{2}\left(\gamma^{-1}+\tau^{-1}+\frac{\gcd(d,2)}{d}(\tau^{-1}-\gamma^{-1})\right) t
                        + c(t)
                \]
                for some periodic function $c(t)$.  Furthermore, we know that the minimal period of $c(t)$ is a divisor of the least integer $D$ such that 
             $D\tau^{-1}$,  $D\gamma^{-1}$,
            and $D\tau\gamma^{-1}$ are integral.  The key observation is that the affine spans of each side of the triangle contain lattice points, so that the periods of the linear and quadratic terms are one.      
        As discussed in \cite{bsw08}, it is a subtle question to determine the minimal period.
 Taking $t = p^{n}$ makes the main term in Theorem~\ref{thm: main theorem} transparent. 
 
        Equation~\eqref{eq:alternateformula} shows that (just as in \Cref{prop: new sum breakdown}) the heart of the proof must be to understand sums involving $\delta(i)$.  This is the subject of \Cref{sec: sums delta}.  We choose to directly study the sums involving floor functions appearing in \Cref{prop: new sum breakdown}, instead of this approach using the Ehrhart quasi-polynomial, as the easy arguments we give in \Cref{sec: sums floor} introduce the techniques we use for studying the sums involving $\delta(i)$ in \Cref{sec: sums delta}.
        \end{remark}

\section{Computing Higher \texorpdfstring{$a$}{a}-Numbers} \label{sec:computinghigheranumbers}
    In this section we analyze the sums appearing in \Cref{prop: new sum breakdown}.  In particular, for a positive rational number $x$ we study sums of the form 
    \[
    \sum_{i=1}^{\floor{ x^{-1} p^n}} (p^n - \floor{xi}) \quad \text{and} \quad  \sum_{i=1}^{\floor{x^{-1} p^n}} \delta(i) .
    \]

    \subsection{Sums involving floor functions} \label{sec: sums floor}
        For a rational number $x$, recall that $\fracpart{x} = x - \floor{x}$ is the fractional part of $x$ and that $L_{x}$, $D_x$, and $\avg{x}$ are defined in \Cref{defn: Lx and Dx} in terms of the base-$p$ decimal expansion of $x$.
        
        \begin{lemma}   \label{lemma: fractionalparts}
            For any $x \in \Qpositive$, \(\fracpart{x p^n}\) is periodic in \(n\) for \(n \geq D_x\) with period $L_x$ and average value $\avg{\fracpart{xp^{n}}} = \avg{x}/(p-1)$.
        \end{lemma}

        \begin{proof}            
            Let $x = \sum x_i p^i$ be the base-$p$ expansion of $x$. Then
            \begin{equation*}
              \fracpart{x p^n } = \sum_{i \leq -1} x_{i -n} p^{i}.
            \end{equation*}
            We know $x_{i-n} = x_{i-n-L_x}$ for $-(i-n) \geq D_x$ + 1 as the base-$p$ expansion of $x$ is eventually periodic.  Thus if $n \geq D_x$, summing over \(i \leq -1\), we see that $\fracpart{xp^n} = \fracpart{xp^{n + L_x}}$. 
            Furthermore,
            \[
                \avg{\fracpart{xp^n}}
                = \frac{1}{L_x} 
                \sum_{n = D_x}^{D_x + L_x - 1} \fracpart{xp^n}
                = \frac{1}{L_x} 
                \sum_{i = D_x + 1}^{D_x + L_x} 
                x_{-i} (p^{-1} + p^{-2} + \ldots) 
                = \frac{\avg{x}}{p-1}. \qedhere
            \]
        \end{proof}
    
        \begin{lemma}   \label{lemma: floors}
            Let $x$ be a positive rational number. Then 
            \begin{enumerate}[(i)]
                \item \(\floor{x p^n} \equiv \floor{x p^{n+L_x}} \pmod{x_N}\) for \(n \geq D_x\).
                \item \(\floor{x p^n} \equiv \floor{x p^{n+L_x}} \pmod{p}\) for \(n \geq D_x + 1\).
            \end{enumerate}
        \end{lemma}
     
        \begin{proof}
            By \Cref{lemma: fractionalparts}, \(\fracpart{x p^{e}} = \fracpart{x p^{e + L_x}}\) for \(e \geq D_x\), and hence
                   \[\floor{x p^{e + L_x}} - \floor{x p^{e}} = x p^{e + L_x} - x p^{e}.\]
            Thus we see that
             \[\floor{x p^{n + L_x}} - \floor{x p^{n}} = p^{n-D_x} x (p^{D_x + L_x} - p^{D_x})
            \]
            which implies that $\floor{x p^{n + L_x}} \equiv \floor{x p^{n}} \pmod{p}$ for \(n \geq D_x + 1\).  Furthermore, for \(n \geq D_x\), the left side shows this quantity is an integer, while the right side shows that it is a multiple of $x_N$ since the denominator $x_D$ is coprime with $x_N$.  Thus $\floor{x p^{n + L_x}} \equiv \floor{x p^{n}} \pmod{x_N}$ for \(n \geq D_x\) as well.
        \end{proof}
    
        \begin{proposition} \label{prop: single floor sum}
            Let $x$ be a positive rational number with $v_p(x) \geq 0$.  Then
            \[
                \sum_{i=1}^{\floor{x^{-1} p^n}} (p^n - \floor{x i}) 
            =
                \frac{1}{2}x^{-1} p^{2n} + \frac{1}{2} \left( x^{-1}\left(1 - \frac{1}{x_D}\right) - 1 \right) p^n + A_{x^{-1}} (n)
            \]
            where the function
            \begin{align}
                A_{x^{-1}}(n)
                \colonequals
                &\frac{1}{2} \left( -\left(1 - \frac{1}{x_D}\right) + x (1 - \fracpart{x^{-1} p^n}) \right) \fracpart{x^{-1} p^n} + \sum_{k=1}^{\floor{x^{-1} p^n} \bmod x_D} \left( \fracpart{xk} - \frac{1}{2} \left(1 - \frac{1}{x_D}\right) \right)
            \end{align}
            is periodic for \(n \geq D_{x^{-1}}\) with period \(L_{x^{-1}}\).
        \end{proposition}
    
        \begin{proof}
            We split the summand using $p^n - \floor{xi} = (p^n - xi) + \fracpart{xi}$.  Note that
            \[
            \sum_{i=1}^{\floor{x^{-1} p^n}} (p^n - xi) = \floor{x^{-1} p^n} p^n - x \cdot \frac{\floor{x^{-1} p^n} \paren{\floor{x^{-1} p^n}+1}}{2}.
            \]
            Rewriting using \(\floor{x^{-1} p^n} = x^{-1} p^n - \fracpart{x^{-1} p^n}\), we see that
                \begin{equation}    \label{single floor sum:eq1}
                    \sum_{i=1}^{\floor{x^{-1} p^n}} (p^n - xi) = 
                    \frac{x^{-1}}{2} p^{2n} - \frac{1}{2} p^n + \frac{x}{2}\paren{1 - \fracpart{x^{-1} p^n}} \fracpart{x^{-1} p^n}.
                \end{equation}
            It remains to consider the sum of the periodic %
            function $\fracpart{xi}$. Since $v_p(x) \geq 0$, the denominator of \(x\) (in lowest terms) is \(x_D\) so $\fracpart{xi}$ is immediately periodic with period $x_D$. The average value of the function is $m \colonequals \frac{1}{2}\paren{1 - \frac{1}{x_D}}$ as in a block of length $x_D$ the fractional part $\fracpart{xi}$ ranges through $0,1/x_D,2/x_D,\dots, (x_D-1)/x_D$ in some order.
            Applying \Cref{lemma: periodic sum}, we obtain
            \begin{align*}
                \sum_{i=1}^{\floor{x^{-1} p^n}} \fracpart{x i}
            &=
               m \floor{x^{-1} p^n} + \sum_{k=1}^{\floor{x^{-1} p^n} \bmod x_D} \left( \fracpart{xk} - m \right) 
            \end{align*}
            Expanding using \(\floor{x^{-1} p^n} = x^{-1} p^n - \fracpart{x^{-1} p^n}\) gives 
            \begin{equation}    \label{single floor sum:eq2}
                \sum_{i=1}^{\floor{x^{-1} p^n}} \fracpart{x i} =
                m  x^{-1} p^n - m  \fracpart{x^{-1} p^n}
              + \sum_{k=1}^{\floor{x^{-1} p^n} \bmod x_D} \paren{ \fracpart{xk} - m }
            \end{equation}
            The desired formula follows from combining equations \eqref{single floor sum:eq1} and \eqref{single floor sum:eq2}. By \Cref{lemma: fractionalparts}, $\fracpart{x^{-1}p^n}$ is periodic for \(n \geq D_{x^{-1}}\) with period $L_{x^{-1}}$. By \Cref{lemma: floors}, \(\paren{\floor{x^{-1} p^n} \bmod x_D}\) is periodic for \(n \geq D_{x^{-1}}\) with period $L_{x^{-1}}$. This gives the desired periodicity for $A_{x^{-1}}(n)$.
        \end{proof}

        \begin{corollary} \label{cor: difference of floor sums}
            With notation as in \Cref{defn: tau and gamma},
                \begin{align*}
                    \sum_{i=1 }^{\floor{\tau^{-1} p^n}} (p^n - \floor{\tau i}) 
                    & -
                    \sum_{i=1}^{\floor{\gamma^{-1} p^n}}  (p^n - \floor{\gamma i})
                \\
                    = &\frac{1}{2} \left( \tau^{-1} - \gamma^{-1} \right) p^{2n}
                    +
                    \frac{1}{2} \left( \tau^{-1} - \gamma^{-1} \right)
                    \left( 1 - \frac{1}{\tau_D} \right) p^n 
                    +
                    \left( A_{\tau^{-1}}(n) - A_{\gamma^{-1}}(n) \right)
                \end{align*}
        \end{corollary}

        \begin{proof}
            Note $v_p(\tau) = 0$ and $v_p(\gamma) \geq 0$ because $d \mid p-1$ and therefore $p\nmid d$. We can then apply \Cref{prop: single floor sum} and rewrite
            \begin{align*}
                &\sum_{i=1 }^{\floor{\tau^{-1} p^n}} (p^n - \floor{\tau i}) 
                -
                \sum_{i=1}^{\floor{\gamma^{-1} p^n}}  (p^n - \floor{\gamma i}) \\
                &= \left[\frac{1}{2}\tau^{-1} p^{2n} + \frac{1}{2} \left( \tau^{-1}\left(1 - \frac{1}{\tau_D}\right) - 1 \right) p^n + A_{\tau^{-1}} (n)\right] \\
                &- \left[\frac{1}{2}\gamma^{-1} p^{2n} + \frac{1}{2} \left( \gamma^{-1}\left(1 - \frac{1}{\gamma_D}\right) - 1 \right) p^n + A_{\gamma^{-1}} (n)\right] \tag{\Cref{prop: single floor sum}} \\
                &=\frac{1}{2} \left( \tau^{-1} - \gamma^{-1} \right) p^{2n}
                +
                \frac{1}{2} \left( \tau^{-1} - \gamma^{-1} \right)
                \left( 1 - \frac{1}{\tau_D} \right) p^n 
                +
                \left( A_{\tau^{-1}}(n) - A_{\gamma^{-1}}(n) \right). \qedhere
            \end{align*}
        \end{proof}

        \begin{example} \label{ex: sum floors p5d4}
            For \(p = 5\), \(d = 4\), and \(r = 2\), 
            \begin{align*}
                 \sum_{i=1}^{\floor{\gamma^{-1} p^n}}  (\gamma - \tau) i + 
                 \sum_{i=\floor{\gamma^{-1} p^n} +1 }^{\floor{\tau^{-1} p^n}} 
                 (p^n - \floor{\tau i})
                =  \frac{4}{21} \cdot 5^{2n}
                +
                \frac{2}{21} \cdot 5^n 
                +
                c^{(1)}(n)
            \end{align*}
            where \(c^{(1)}(n)\) is given in \Cref{table: sum floors p5d4 quasiconstant}.

            \begin{table}[hb]   \label{table: sum floors p5d4 quasiconstant}
                \begin{tabular}{|l|p{1cm}|p{1cm}|p{1cm}|p{1cm}|p{1cm}|p{1cm}|}
                    \hline
                    \(n \pmod{6}\)     & \(0\) & \(1\) & \(2\) & \(3\) & \(4\) & \(5\) \\ \hline
                    \(c^{(1)}(n)\)
                    & \(-2/7\) & \(-5/21\) & \(-3/7\)
                    & \(-2/21\) & \(-2/7\) & \(1/3\)                \\ \hline
                \end{tabular}
                \caption{Value of periodic function \(c^{(1)}(n)\) for \(n \geq 0.\)}
            \end{table}
        \end{example}

        \subsection{Sums involving \texorpdfstring{$\delta$}{\textbackslash delta}}\label{sec: sums delta}
        We now turn to the sums of $\delta(i)$ appearing in \Cref{prop: new sum breakdown}. 
        We first reduce to considering the case where $i$ is prime to $p$.

        \begin{lemma} \label{lemma: delta sum by exponent}
            Let \(x \in \Qpositive\) with \(x \geq 1\). Then
            \[ \sum_{i=1}^{\floor{x^{-1} p^n}} \delta(i) 
            = 
            \sum_{e=0}^n 
            \sum_{\substack{i=1 \\ p \: \nmid \: i}}^{\floor{x^{-1} p^e}} 
            \delta(i) 
            =
            \sum_{e=0}^n \sum_{i=1}^{\floor{x^{-1} p^e}} \delta_0(i).\]
        \end{lemma}
    
        \begin{proof}
            This is an immediate consequence of \Cref{thm: delta}(\ref{thmparti}) and rewriting using \Cref{defn: deltas}.
        \end{proof}

        As $\delta_0(i)$ is periodic, the inner sum is easy to evaluate.

        \begin{lemma} \label{lemma: delta sum for exponent}
           Let \(x \in \Qpositive\) and $e$ be a non-negative integer. Then
            \begin{align*}
                \sum_{i=1}^{\floor{x^{-1} p^e}} \delta_0(i)
                =
                &\frac{1}{2} 
                \left( 1 - \frac{1}{p} \right)
                \left( 1 - \frac{1}{\tau_D} \right)
                \left( x^{-1} p^e - \fracpart{x^{-1} p^e} \right) 
                + F_{x^{-1}}(e)
            \end{align*}
            where
            \begin{align*}
                F_{x^{-1}}(e) \colonequals 
                \sum_{k = 1}^{\floor{x^{-1} p^e} \bmod{\tau_Dp}} 
                \left(
                \delta_0(k)
                -
                \frac{1}{2} 
                \left( 1 - \frac{1}{p} \right)
                \left( 1 - \frac{1}{\tau_D} \right)
                \right).
            \end{align*}
        \end{lemma}
    
        \begin{proof}
        \Cref{thm: delta}(\ref{thmpartii}) shows that \(\delta_0\) is an immediately periodic function with period \(\tau_Dp\), so we may apply \Cref{lemma: periodic sum}. \Cref{prop: number of subpalindromics in a tau_Dp interval} shows that the average value of $\delta_0$ over a block of length
            $\tau_Dp$ is \(\frac{1}{2} \left( 1 - \frac{1}{p} \right) \left( 1 - \frac{1}{\tau_D} \right)\).
        \end{proof}

        \begin{corollary}   \label{cor: delta sum over exponents}
            Let \(x \in \Qpositive\) and $n$ be a positive integer with \(x \geq 1\). Then we have
            \[
                \sum_{i=1}^{\floor{x^{-1} p^n}} \delta(i) = 
                \frac{1}{2}
                \left( 1 - \frac{1}{\tau_D} \right)
                \left( p^n - \frac{1}{p} \right) x^{-1}
                - \frac{1}{2} \left( 1 - \frac{1}{p} \right) \left( 1 - \frac{1}{\tau_D} \right) \sum_{e=0}^n \fracpart{x^{-1} p^e}
                + \sum_{e=0}^n F_{x^{-1}}(e).
           \]
        \end{corollary}

        \begin{proof}
            We combine \Cref{lemma: delta sum by exponent} with \Cref{lemma: delta sum for exponent} to obtain
            \begin{align*}
                \sum_{i=1}^{\floor{x^{-1} p^n}} \delta(i) &= 
                \sum_{e=0}^n \sum_{i=1}^{\floor{x^{-1} p^e}} \delta_0(i) \tag{\Cref{lemma: delta sum by exponent}} \\
                &= \sum_{e=0}^n \left[ \frac{1}{2} 
                \left( 1 - \frac{1}{p} \right)
                \left( 1 - \frac{1}{\tau_D} \right)
                \left( x^{-1} p^e - \fracpart{x^{-1} p^e} \right) 
                + F_{x^{-1}}(e)  \right]\tag{\Cref{lemma: delta sum for exponent}} \\
                &= \sum_{e=0}^n \frac{1}{2} 
                \left( 1 - \frac{1}{p} \right)
                \left( 1 - \frac{1}{\tau_D} \right)
                \left( x^{-1} p^e\right) - \sum_{e=0}^n \frac{1}{2} 
                \left( 1 - \frac{1}{p} \right)
                \left( 1 - \frac{1}{\tau_D} \right) \fracpart{x^{-1} p^e} + \sum_{e=0}^n F_{x^{-1}}(e) \\
                &= \frac{1}{2} 
                \left( 1 - \frac{1}{p} \right)
                \left( 1 - \frac{1}{\tau_D} \right) x^{-1} \sum_{e=0}^n p^e - \frac{1}{2} 
                \left( 1 - \frac{1}{p} \right)
                \left( 1 - \frac{1}{\tau_D} \right) \sum_{e=0}^n \fracpart{x^{-1} p^e} + \sum_{e=0}^n F_{x^{-1}}(e).
            \end{align*}
            We simplify the geometric series:
            \begin{align*}
                \frac{1}{2} 
                \left( 1 - \frac{1}{p} \right)
                \left( 1 - \frac{1}{\tau_D} \right) x^{-1} \sum_{e=0}^n p^e 
                &= \frac{1}{2} 
                \left( 1 - \frac{1}{p} \right)
                \left( 1 - \frac{1}{\tau_D} \right) x^{-1} \frac{p^{n+1} -1}{p-1} \\
                &= \frac{1}{2} 
                \left( 1 - \frac{1}{\tau_D} \right) \left(p^{n} -\frac{1}{p}\right) x^{-1}.
            \end{align*}
            Therefore we conclude
            \begin{align*}
                \sum_{i=1}^{\floor{x^{-1} p^n}} \delta(i) 
                &= 
                \frac{1}{2}
                \left( 1 - \frac{1}{\tau_D} \right)
                \left( p^n - \frac{1}{p} \right) x^{-1}
                - \frac{1}{2} \left( 1 - \frac{1}{p} \right) \left( 1 - \frac{1}{\tau_D} \right) \sum_{e=0}^n \fracpart{x^{-1} p^e}
                + \sum_{e=0}^n F_{x^{-1}}(e). \qedhere
            \end{align*}
        \end{proof}

        \begin{proposition} \label{prop: delta sum}
            Suppose that $x_D = \tau_D$. Then \(F_{x^{-1}}(e)\) is periodic for \(e \geq D_{x^{-1}} + 1\) with period $L_{x^{-1}}$. Moreover, for $n \geq D_{x^{-1}}$ we have
            \begin{align*}
                \sum_{i=1}^{\floor{x^{-1} p^n}} \delta(i) 
                &= \frac{1}{2} x^{-1}
                \left( 1 - \frac{1}{\tau_D} \right) p^n
                +
                \left(
                \avg{F_{x^{-1}}} 
                - \frac{\avg{x^{-1}}}{2p} 
                \left( 1 - \frac{1}{\tau_D}\right) 
                \right) n 
                + B_{x^{-1}}(n)
            \end{align*}
            where \(B_{x^{-1}}(n)\) is a periodic function for \(n \geq D_{x^{-1}}\) with period \(L_{x^{-1}}\) defined in \eqref{eq:Bdef}.
        \end{proposition}
        
        \begin{proof}
            Note that $F_{x^{-1}}(n)$ depends only on the value of $\lfloor x^{-1} p^n \rfloor$ modulo $\tau_Dp$. By \Cref{lemma: floors}, $\lfloor x^{-1} p^n \rfloor$ is periodic modulo $\tau_Dp$ for \(n \geq D_{x^{-1}} + 1\) with period $L_{x^{-1}}$. This gives the first claim.
            
            By \Cref{cor: eventually periodic sum}, for all $n \geq D_{x^{-1}}$ we have
            \begin{equation*}
                \sum_{e = 0}^n F_{x^{-1}}(e)
                = 
                \sum_{e = 0}^{D_{x^{-1}}} F_{x^{-1}}(e) + \avg{F_{x^{-1}}}(n-D_{x^{-1}}) + 
                \sum_{k = 1}^{(n - D_{x^{-1}}) \bmod L_x} (F_{x^{-1}}(k + D_{x^{-1}})-\avg{F_{x^{-1}}}) 
                = 
                \avg{F_{x^{-1}}} \cdot n + B_{x^{-1}}'(n),
            \end{equation*}
            where $B_{x^{-1}}'(n)$ is periodic for \(n \geq D_{x^{-1}}\) with period $L_{x^{-1}}$. Note that the above application of \Cref{cor: eventually periodic sum} is slightly different in that the 0-indexed term \(F_{x^{-1}}(0)\) is included in the summation. However, this term is excluded from the repeating portion of the sequence and is subsumed by \(B_{x^{-1}}'(n)\).
    
            On the other hand, by \Cref{lemma: fractionalparts} we know that $\{x^{-1} p^e\}$ is periodic for \(e \geq D_{x^{-1}}\) with period $L_{x^{-1}}$ and average value $\avg{x^{-1}}/(p-1)$. By \Cref{cor: eventually periodic sum}, for all $n \geq D_{x^{-1}} - 1$ we have 
            \begin{align*}
                \sum_{e = 0}^n \fracpart{x^{-1} p^e}
                &= 
                \sum_{e = 0}^{D_{x^{-1}}-1} \fracpart{x^{-1} p^e} + 
                \frac{\avg{x^{-1}}}{p-1}(n-D_{x^{-1}}+1) + 
                \sum_{k = 1}^{(n - D_{x^{-1}} + 1) \mod L_{x^{-1}}} \left(\fracpart{x^{-1} p^k} - 
                \frac{\avg{x^{-1}}}{p-1}\right) 
            \\
                &= \frac{\avg{x^{-1}}}{p-1} \cdot n + B_{x^{-1}}''(n),
            \end{align*}
            where \(B_{x^{-1}}''(n)\) is periodic for \(n \geq D_{x^{-1}} - 1\) with period \(L_{x^{-1}}\). Once again, \(B_{x^{-1}}''(n)\) subsumes the 0-indexed term. To complete the proof we apply \Cref{cor: delta sum over exponents} and set
            \begin{equation} \label{eq:Bdef}
                B_{x^{-1}}(n) \colonequals 
                B'_{x^{-1}}(n) - 
                \frac{1}{2} \left( 1 - \frac{1}{p} \right) \left( 1 - \frac{1}{\tau_D} \right) B''_{x^{-1}} (n) - \frac{1}{2p}\left (1 - \frac{1}{\tau_D} \right )x^{-1}. \qedhere
            \end{equation}
        \end{proof}

        \begin{example} \label{ex: deltasum p5d4}
            Let $p=5$, $d=4$, and \(r = 2\).  Then we have       
            \begin{align*}
                \sum_{i=1}^{\floor{\tau^{-1} p^n}} \delta(i) 
                &= 
                \frac{1}{6} \cdot 5^n +
                \begin{cases}
                        -\frac{1}{6} & \text{if \(n\) is even} \\
                        \phantom{-}\frac{1}{6} & \text{if \(n\) is odd}
                    \end{cases}
            \\
                \sum_{i=1}^{\floor{\gamma^{-1} p^n}} \delta(i) 
                &=
                \frac{1}{14} \cdot 5^n + \frac{1}{3} n +
                c^{(2)}(n)
            \end{align*}
            where \(c^{(2)}(n)\) is given in \Cref{table: delta sum p5d4 quasiconstant}.  
            \begin{table}[H]    \label{table: delta sum p5d4 quasiconstant}
                \begin{tabular}{|l|p{1cm}|p{1cm}|p{1cm}|p{1cm}|p{1cm}|p{1cm}|}
                    \hline
                    \(n \pmod{6}\)     & \(0\) & \(1\) & \(2\) & \(3\) & \(4\) & \(5\) \\ \hline
                    \(c^{(2)}(n)\)
                    & \(-1/14\) & \(13/42\) & \(23/42\)
                    & \(1/14\) & \(1/42\) & \(5/42\)                \\ \hline
                \end{tabular}
                \caption{Value of periodic function \(c^{(2)}(n)\) for \(n \geq 0.\)}
            \end{table}    
        \end{example}

        \begin{remark}
        Note the formulas in this subsection appears quite similar to ``classical'' formulas in Iwasawa theory, i.e. of the form $\mu p^n + \lambda n + \nu$,
with the only difference being that $\nu$ is a periodic function of $n$.  In contrast, the formulas for the sums in Section~\ref{sec: sums floor} are of the form $\alpha p^{2n}+ \beta p^{n} + \gamma$ with $\gamma$ periodic: ``classical'' formulas in Iwasawa theory do not include a $p^{2n}$ term.
        \end{remark}

    \subsection{Proof of the Main Theorem} \label{sec: main thm}
        We will now prove \Cref{thm: main theorem} by combining \Cref{prop: new sum breakdown} with \Cref{cor: difference of floor sums,cor: difference of delta sums} and attempting to simplify the resulting formula.
        The main simplification comes from the following.
   
        \begin{proposition} \label{prop: linear coefficient of tau inverse sum is zero}
            The coefficient of \(n\) in \Cref{prop: delta sum} is zero when $x = \tau$. That is,
            \[\avg{F_{\tau^{-1}}} - \frac{\avg{\tau^{-1}}}{2p} \left( 1 - \frac{1}{\tau_D}\right) = 0.\]
        \end{proposition}
        
        \begin{proof}
            As $d < p $, we see that the base-$p$ expansion of $\tau^{-1} = \frac{d}{p+1}$ is
            \begin{equation}    \label{eq: basep tau}
                \frac{d}{p+1} = 
                (d-1)p^{-1} + (p-d)p^{-2} + (d-1) p^{-3} + (p-d) p^{-4} + \dots.
            \end{equation}
            Thus $\avg{\tau^{-1}} = \frac{1}{2}((d-1)+(p-d)) = (p-1)/2$, so that
            \begin{equation*}
                \frac{\avg{\tau^{-1}}}{2p} \left ( 1 - \frac{1}{\tau_D} \right) = \frac{1}{4} \left(1 - \frac{1}{p} \right) \left( 1-\frac{1}{\tau_D} \right).
            \end{equation*}

            To evaluate $\avg{F_{\tau^{-1}}} = \frac{1}{2} ( F_{\tau^{-1}}(1) + F_{\tau^{-1}}(2))$, note that $( \floor{\tau^{-1} p^e} \bmod p)$ is $d-1$ if $e$ is a positive odd integer and is $p-d$ if $e$ is a positive even integer.  Similarly, $( \floor{\tau^{-1} p^e} \bmod \tau_D)$ is $\tau_D-1$ if $e$ is a positive odd integer and $0$ if $e$ is a positive even integer. The Chinese Remainder Theorem gives 
            \[
            \paren{ \floor{\tau^{-1} p^e} \bmod \tau_Dp } = 
            \begin{cases}
                0 & \text{if } e = 0                                    \\
                d - 1 & \text{if \(e\) is odd}                          \\
                \tau_Dp - d & \text{if \(e\) is even and \(e > 0\)}.
            \end{cases}
            \]
            Recalling the definition of $F_{x^{-1}}(n)$ from \Cref{lemma: delta sum for exponent}, we see
            \begin{align*}
                F_{\tau^{-1}}(1) &=  \sum_{i=1}^{d-1} \left ( \delta_0(i) - \frac{1}{2} \left( 1 - \frac{1}{p} \right) \left (1 - \frac{1}{\tau_D} \right) \right )     \\
                F_{\tau^{-1}}(2) & =  \sum_{i=1}^{\tau_Dp-d} \left ( \delta_0(i) - \frac{1}{2} \left( 1 - \frac{1}{p} \right) \left (1 - \frac{1}{\tau_D} \right) \right ).
            \end{align*}
            Thus we are led to consider the sum
            \begin{equation} 
            \sum_{i=1}^{d-1} \delta_0(i) + \sum_{i=1}^{\tau_Dp-d} \delta_0(i).
            \end{equation}
            By \Cref{prop: delta negation reduction}, for \(1 \leq i \leq d-1\), we have \(\delta_0(d - i) = \delta_0(\tau_Dp - i)\). Thus
            \begin{equation}
                \sum_{i=1}^{d-1} \delta_0(i) = \sum_{i=1}^{d-1} \delta_0(d - i) = \sum_{i=1}^{d-1} \delta_0(\tau_Dp - i) = \sum_{i=\tau_Dp-d+1}^{\tau_Dp-1} \delta_0(i),
            \end{equation}
            so by \Cref{prop: number of subpalindromics in a tau_Dp interval},
            \begin{equation}    \label{eq: key delta sum}
                \sum_{i=1}^{d-1} \delta_0(i) + \sum_{i=1}^{\tau_Dp-d} \delta_0(i) = \sum_{i=\tau_Dp-d+1}^{\tau_Dp-1} \delta_0(i) + \sum_{i=1}^{\tau_Dp - d} \delta_0(i) = 
                \sum_{i=1}^{\tau_Dp-1} \delta_0(i) = \frac{(p-1)(\tau_D-1)}{2}.
            \end{equation}
            Thus we see that 
            \begin{align*}
                \frac{1}{2} (F_{\tau^{-1}}(1) + F_{\tau^{-1}}(2)) &=
                \frac{1}{2} \left( \left(\sum_{i=1}^{d-1} \delta_0(i) + \sum_{i=1}^{\tau_Dp-d} \delta_0(i)  \right) - \frac{1}{2} (\tau_Dp -1) \left(1 - \frac{1}{p} \right) \left( 1-\frac{1}{\tau_D} \right) \right)
            \\
                &= \frac{1}{4} \left(1 - \frac{1}{p} \right) \left( 1-\frac{1}{\tau_D} \right) =  \frac{\avg{\tau^{-1}}}{2p} \left( 1 - \frac{1}{\tau_D}\right). \qedhere
            \end{align*}
        \end{proof}

        \begin{remark}
           We do not have a conceptual explanation for why the coefficient of $n$ is zero in \Cref{prop: linear coefficient of tau inverse sum is zero}.  
        \end{remark}

        \begin{corollary}  \label{cor: difference of delta sums}
            With $B_{x^{-1}}(n)$ as in \Cref{prop: delta sum},
            \begin{align*}
                \sum_{i=\floor{\gamma^{-1} p^n} +1 }^{\lfloor  \tau^{-1} p^n \rfloor} \delta(i)
                = \frac{1}{2} \left(  \tau^{-1} - \gamma^{-1} \right) \left( 1 - \frac{1}{\tau_D} \right) p^n
                +
                \left( 
                \frac{\avg{\gamma^{-1}}}{2p} \left( 1 - \frac{1}{\tau_D}\right)
                -
                \avg{F_{\gamma^{-1}}}
                \right) n
               + \paren{B_{\tau^{-1}}(n) - B_{\gamma^{-1}}(n)}.
            \end{align*}
        \end{corollary}
        \begin{proof}
            Combine \Cref{prop: delta sum} (as $\tau_D = \gamma_D$) and \Cref{prop: linear coefficient of tau inverse sum is zero}.
        \end{proof}
    
        We now synthesize all our previous results to arrive at a version of our main theorem.
        
        \begin{theorem} \label{thm: main v.1}
            With the standard setup, there exists a function $\nu_r: \N \to \Q$ such that
            \begin{align*}
                a^r(X_n) 
                &=
                \frac{1}{2} \left( \tau^{-1} - \gamma^{-1} \right) p^{2n} + \lambda_r n + \nu_r(n)
            \end{align*}
            where 
            \[\lambda_r =
            \avg{F_{\gamma^{-1}}} -
            \frac{\avg{\gamma^{-1}}}{2p} \left( 1 - \frac{1}{\tau_D}\right)
            \]
              and \(\nu_r(n)\) is periodic for \(n \geq D_{\gamma^{-1}}\) with minimal period $\lcm (L_{\gamma^{-1}},2)$ or $\frac{1}{2} \cdot \lcm (L_{\gamma^{-1}}, 2)$.
        \end{theorem}

        Recall that \Cref{lemma: Lx} shows that $D_{\gamma^{-1}} = -v_p(\gamma^{-1})$ and $L_{\gamma^{-1}}$ is the order of $p$ modulo $\gamma_N$.  Whether the minimal period of $\nu_r(n)$ is \(\lcm \left( L_{\gamma^{-1}}, 2 \right)\) or $\frac{1}{2} \cdot \lcm (L_{\gamma^{-1}}, 2)$ appears to be subtle: see \Cref{ex:what is L}.
    
        \begin{proof}
            Start with \Cref{prop: new sum breakdown}, and use
            \Cref{cor: difference of floor sums} for an expression for the sums involving floor functions and \Cref{cor: difference of delta sums} for the sums involving $\delta(i)$.
            Collecting term gives the formula, with 
            \begin{equation}
                \nu_r(n) \colonequals A_{\tau^{-1}}(n) - A_{\gamma^{-1}}(n) - B_{\tau^{-1}}(n) + B_{\gamma^{-1}}(n).
            \end{equation}
           Recall that $A_{x^{-1}}(n)$ and $B_{x^{-1}}(n)$ are periodic for $n \geq D_{x^{-1}}$ with period $L_{x^{-1}}$ (\Cref{prop: single floor sum,prop: delta sum}).
            Thus \(\nu_r(n)\) is periodic for \(n \geq \max\left(D_{\tau^{-1}}, D_{\gamma^{-1}}\right)\) with period 
          \(   \lcm (L_{\tau^{-1}}, L_{\gamma^{-1}}).\)
            Equation \eqref{eq: basep tau} shows that \(L_{\tau^{-1}} \leq 2\) and \(D_{\tau^{-1}} = 0\), which gives the claimed delay and \(\lcm \left( L_{\gamma^{-1}}, 2 \right)\) as a bound on the minimal period.

            We now show the minimal period is either \(\lcm \left( L_{\gamma^{-1}}, 2 \right)\) or half of it. We have that there is $\lambda_r \in \Q$ and $c : \N \to \Q$ such that for sufficiently large $n$,
                \begin{align*}
                    a^r(X_{n+L}) - a^r(X_n) 
                    &= \frac{1}{2}(\tau^{-1} - \gamma^{-1})p^{2(n+L)} + \lambda_r (n+L) + c(n + L) \\
                    &- \left( \frac{1}{2}(\tau^{-1} - \gamma^{-1})p^{2n} + \lambda_r n + c(n) \right)\\ 
                    &= \frac{1}{2}(\tau^{-1} - \gamma^{-1})(p^{2L} - 1)p^{2n} + \lambda_r L \\
                    &= \frac{1}{2}\frac{d}{p+1}(p^{2L} - 1)p^{2n} - \frac{1}{2}\gamma^{-1}(p^{2L} - 1)p^{2n} + \lambda_r L
                \end{align*}
            The left side is clearly an integer. Note that $\frac{1}{2}\frac{d}{p+1}(p^{2L} - 1)p^{2n}$ is an integer as well:  $p+1$ must divide one of $p^L-1$ or $p^L+1$ and the other is a multiple of $2$ (unless $p=2$, in which case $2 \mid p^{2n}$).  
            Similarly, we note that $(p^{2L} - 1)p^{2n} $ is even.  Furthermore, note that if we clear denominators by multiplying both sides by $\gamma_N$, then $\frac{1}{2} \gamma_N \gamma^{-1} (p^{2L} - 1)p^{2n}$ is an integer.  Thus we can conclude that $\gamma_N \lambda_r L$ is an integer. %
            Furthermore, reducing modulo $\gamma_N$ shows that
            \begin{equation*}
                0 \equiv -\frac{\gamma_D}{2}(p^{2L}-1)p^{2n-v_p(\gamma)} + \gamma_N \lambda_r L \pmod{\gamma_N}
            \end{equation*}
            As this holds for all sufficiently large $n$, we see that $p^{2L}-1 \equiv 0 \pmod{\gamma_N}$. In other words, by \Cref{lemma: Lx}, $L_{\gamma^{-1}} \mid 2L$.  As $L \mid \lcm(L_{\gamma^{-1}}, 2)$ by \Cref{thm: main v.1}, we see $L = \lcm(L_{\gamma^{-1}}, 2)$ or $L = \frac{1}{2}\lcm(L_{\gamma^{-1}}, 2)$. 
        \end{proof}
        
        \begin{example} \label{ex: final p5d4}
        We can now establish \Cref{ex: intro p5d4}, which has \(p = 5\), \(d=4\), and \(r=2\).  Combining \Cref{thm: main v.1} (or really just \Cref{prop: new sum breakdown}) with \Cref{ex: deltasum p5d4,ex: sum floors p5d4} we see that
        \[a^2(X_n) = \frac{4}{21} \cdot 5^{2n} + \frac{1}{3}n + \nu_2(n)\]
        where \(\nu_2(n)\) is given in \Cref{table: total p4d4 quasiconstant}. Note that the period given by the theorem (6) is not the minimal period (3), which we observe in this instance coincides with \(L_{\gamma^{-1}}\).
        \end{example}

        \begin{corollary}   \label{cor: d1d2}
            When $d=1 $ or $d=2$:
            \begin{enumerate}[(i)]
                \item $\lambda_r = 0$.
                \item The minimal period of $a^r(X_n)$ is
                $$ L =
                \begin{cases}
                    L_{\gamma^{-1}}  & \text{ if } L_{\gamma^{-1}} \text{ is odd} \\
                    L_{\gamma^{-1}} \text{ or } \frac{L_{\gamma^{-1}}}{2} & \text{ if } L_{\gamma^{-1}} \text{ is even.} \\
                \end{cases}
                $$
            \end{enumerate}
        \end{corollary}

        \begin{proof}
            When \(d=1\) or \(d = 2\), then $\tau_D =1$ (recall \Cref{defn: ND notation}), so \(\delta(i) = 0\) for all integers \(i\). Hence, by \Cref{prop: new sum breakdown} and \Cref{cor: difference of floor sums},
            \begin{align*}
                a^r(X_n) &= 
                \sum_{i=1 }^{\floor{\tau^{-1} p^n}} 
                (p^n - \floor{\tau i}) 
                -
                \sum_{i=1}^{\floor{\gamma^{-1} p^n}}  (p^n - \floor{\gamma i})
            \\
                &=
                \frac{1}{2} \left( \tau^{-1} - \gamma^{-1} \right) p^{2n}
                +
                \frac{1}{2} \left( \tau^{-1} - \gamma^{-1} \right) \left( 1 - \frac{1}{\tau_D} \right)p^n 
                +
                \left( A_{\tau^{-1}}(n) - A_{\gamma^{-1}}(n) \right)
            \\
                &=
                \frac{1}{2} \left( \tau^{-1} - \gamma^{-1} \right) p^{2n}
                +
                \left( A_{\tau^{-1}}(n) - A_{\gamma^{-1}}(n) \right),
            \end{align*}
            showing \(\lambda_r = 0\). Furthermore, noting that \(\floor{\tau^{-1} p^n} \bmod \tau_D\) is \(0\) for \(\tau_D = 1\), \Cref{prop: single floor sum} gives
            \begin{align*}
                    A_{\tau^{-1}}(n)
                \colonequals
                    &\frac{1}{2} \left( -\left(1 - \frac{1}{\tau_D}\right) + \tau (1 - \fracpart{\tau^{-1} p^n}) \right) \fracpart{\tau^{-1} p^n} + \sum_{k=1}^{\floor{\tau^{-1} p^n} \bmod \tau_D} \left( \fracpart{\tau k} - \frac{1}{2} \left(1 - \frac{1}{\tau_D}\right) \right) \\
                     =& \frac{1}{2}\tau (1 - \fracpart{\tau^{-1} p^n}) \fracpart{\tau^{-1} p^n} .
            \end{align*}
            Notice that $1 - \fracpart{\tau^{-1}p^n} = \fracpart{\tau^{-1}p^{n+1}} = 1 - \fracpart{\tau^{-1}p^{n+2}}$. Hence, \(A_{\tau^{-1}}(n) = A_{\tau^{-1}}(n + 1)\), so $A_{\tau^{-1}}(n)$ is constant. Recalling that $A_{\gamma^{-1}}$ has period $L_{\gamma^{-1}}$, we conclude that \(A_{\tau^{-1}}(n) - A_{\gamma^{-1}}(n)\) and \(a^r(X_n)\) have period at most \(L_{\gamma^{-1}}\). \Cref{thm: main v.1} then yields the desired period.
 \end{proof}

\begin{example} \label{ex:what is L}
    We can observe the above corollary in \Cref{table: L1}, which shows the behavior of the period when $d = 2$, $p=5$ for $r = 1$ to $r = 16$.
        \begin{table}[ht]
            \begin{tabular}{l|llllllllllllllll}
            $r$ & $1$ & $2$ & $3$ & $4$ & $5$ & $6$ & $7$ & $8$ & $9$ & $10$ & $11$ & $12$ & $13$ & $14$ & $15$ & $16$ \\ \hline 
            $L$ & $1$ & $3$ & $3$ & $5$ & $2$ & $1$ & $8$ & $9$ & $3$ & $11$ & $1$ & $9$ & $7$ & $3$ & $10$ & $3$ \\
            $L_{\gamma^{-1}}$ & $1$ & $6$ & $6$ & $5$ & $4$ & $2$ & $16$ & $9$ & $6$ & $22$ & $1$ & $18$ & $14$ & $3$ & $10$ & $6$ \\
            $\lcm(L_{\gamma^{-1}}, 2)$ & $2$ & $6$ & $6$ & $10$ & $4$ & $2$ & $16$ & $18$ & $6$ & $22$ & $2$ & $18$ & $14$ & $6$ & $10$ & $6$  \
            \end{tabular}
            \caption{Comparing $L$ to $\lcm(L_{\gamma^{-1}}, 2)$ with $d=2$, $p=5$.}
            \label{table: L1}
            \end{table}
\end{example}

        We can also obtain an explicit formula in the following case: 
        
        \begin{corollary}   \label{cor: r = p + 1}
            When $r = p + 1$, we have $\lambda_r = 0$, the minimal period $L$ is $1$, and $\nu_r(n) = \frac{1}{2} (p-1) (\tau^{-1} -1)$ for $n \geq 1$.
        \end{corollary}

        \begin{proof}
            For \(r = p+1\), we have $\gamma^{-1}=\frac{1}{p}\tau^{-1}$, so $\avg{F_{\gamma^{-1}}}=\avg{F_{\tau^{-1}}}$ and $\avg{\gamma^{-1}} = \avg{\tau^{-1}}$. Thus, by \Cref{thm: main v.1} and \Cref{prop: linear coefficient of tau inverse sum is zero}, we obtain \(\lambda_r = 0\).  It is also immediate that $D_{\gamma^{-1}} = 1$.
    
            To see the minimal period, we will analyze
             \begin{equation*}
                 \nu_r(n) = A_{\tau^{-1}}(n) - A_{\gamma^{-1}}(n) + B_{\gamma^{-1}}(n) - B_{\tau^{-1}}(n).
             \end{equation*}
             This is slightly tricky, as $A_{\tau^{-1}}(n) - A_{\gamma^{-1}}(n)$ and $B_{\gamma^{-1}}(n) - B_{\tau^{-1}}(n)$ both have period two but the sum is constant. It is an elementary but involved exercise, which we leave to the reader, to check
            \begin{equation*}
                A_{\tau^{-1}}(n) - A_{\gamma^{-1}}(n) 
                = 
                \frac{1}{2}(p-1)(\tau^{-1} - 1) 
                + (-1)^n \paren{-\paren{1 - \frac{1}{\tau_D}}\tau^{-1}}    \quad \text{and}     
            \end{equation*}
            
            \begin{align*}
                B_{\gamma^{-1}}(n) - B_{\tau^{-1}}(n) 
                &= B_{\tau^{-1}}(n - 1) - B_{\tau^{-1}}(n)
            \\
                &= (-1)^n \paren{ \frac{1}{2}\paren{1 - \frac{1}{\tau_D}}\paren{\frac{1}{p} d + \paren{1 - \frac{1}{p}} \tau^{-1} } }.
            \end{align*}
             There are two key observations:
             \begin{itemize}
                 \item  $\gamma^{-1}=\frac{1}{p}\tau^{-1}$, so for example $B_{\gamma^{-1}}(n) = B_{\tau^{-1}}(n-1)$.
                 \item  as $\floor{\tau^{-1} p^n} \mod{\tau_D}$ is $0$ if $n$ is even and $\tau_D-1$ if $n$ is odd, the sums appearing in the definitions of $A_{\tau^{-1}}(n)$ and $B_{\tau^{-1}}(n)$ are almost the same, with one summand depending on the parity of $n$.
             \end{itemize}
             It is then a long but elementary computation to work out the exact formula.
       \end{proof}     
      
    \subsection{Additional Information about Periods and \texorpdfstring{$\lambda_r$}{lambda}}  \label{sec: additional}
        In this section, we tighten our predictions about the parameter $\lambda_r$.

        \begin{corollary} \label{cor: lambda * L_gamma-1 integral}
            The quantity $\lambda_r \cdot L$ is an integer.
        \end{corollary}
        \begin{proof}
            The proof of \Cref{thm: main v.1} shows \(\gamma_N \lambda_r L \equiv 0 \bmod \gamma_N\), which is equivalent to $\lambda_r L \in \mathbb{Z}$.
        \end{proof}

        This is supported by a variety of examples we have investigated. The following proposition relates the formulas for $a^r(X_n)$ for certain pairs of $r$ values.
        
        \begin{proposition}
            Let $r_0,r_1$ be non-negative integers such that $r_1 = (r_0+1)p+1$.  Let $\lambda_0, \lambda_1$ (resp. $D_0, D_1$) be the values of $\lambda_r$ (resp. $D_{\gamma^{-1}}$) for $r_0,r_1$.  Then $\lambda_1 = \lambda_{0}$ and $D_{1} = D_{0}+1$.  %
        \end{proposition}

        \begin{proof}
            Let $\gamma_1$ and $\gamma_0$ be the value of $\gamma = \frac{p-1}{d} r + \tau = \frac{(p-1)r + (p+1)}{d}$ when $r = r_1$ and $r = r_0$ respectively.  We see that
            \[\gamma_1^{-1} = \frac{d}{(p-1)((r_0 +1)p + 1) + p+ 1} = \frac{d}{p ( (p-1) r_0 + (p-1) ) + 2p} = \frac{1}{p} \gamma_0^{-1}\]
             This implies that the repeating parts of the base-$p$ expansions are the same and $D_{\gamma_1^{-1}} = D_{\gamma_0^{-1}}+1$.  Thus $\avg{\gamma_1^{-1}} = \avg{\gamma_0^{-1}}$ and $F_{\gamma_1^{-1}}(e) = F_{\gamma_0^{-1}}(e+1)$. Since the latter implies $\avg{F_{\gamma_1}^{-1}} = \avg{F_{\gamma_0}^{-1}}$, the claim follows from the formula for $\lambda_r$ in \Cref{thm: main v.1}.
        \end{proof}

        In the above proof, \(\gamma_1^{-1} = \frac{1}{p} \gamma_0^{-1}\) implies $L_{\gamma_1^{-1}} = L_{\gamma_0^{-1}}$, giving a relationship between the possible minimal periods of $a^{r_1}(X_n)$ and $a^{r_0}(X_n)$. We conjecture:
        
        \begin{conj}
            For two $r$-values, $r_1$, $r_0$, where $r_1 = (r_0 + 1)p + 1$, the minimal periods for the $a^{r_1}(X_n)$ and $a^{r_0}(X_n)$ are the same.
        \end{conj}

\bibliographystyle{amsalpha}
\bibliography{latticebib}

\end{document}